\documentclass[11pt]{llncs}
\usepackage[a4paper,hmargin=2.5cm,vmargin=2.5cm]{geometry}

\usepackage[hyphens]{url}
\usepackage{amssymb,amsmath}
\usepackage{graphicx}
\usepackage{caption}

\usepackage[hidelinks]{hyperref}

\usepackage{tikz}
\usetikzlibrary{arrows}
\usetikzlibrary{matrix}

\usepackage{enumerate}

\newcommand{\F}{{\mathbb F}}

\newcommand{\Z}{{\mathbb Z}}
\newcommand{\PP}{{\mathbb P}}

\DeclareMathOperator{\PGL}{PGL}
\DeclareMathOperator{\PSL}{PSL}
\DeclareMathOperator{\Gal}{Gal}

\newcommand{\Fo}{{{\mathbb F}_q}}
\newcommand{\Ft}{{{\mathbb F}_{q^2}}}
\newcommand{\matt}[4]{{\begin{pmatrix} {#1} & {#2} \\ {#3} & {#4} \end{pmatrix}}}

\newcommand{\Fqk}{{{\mathbb F}_{q^{k}}}}
\newcommand{\ZNZ}{{{\mathbb Z}/N{\mathbb Z}}}

{\bfseries}{\itshape}

\begin{document}

\title{On the discrete logarithm problem in\\ finite fields of fixed characteristic}
 
\author{Robert Granger\inst{1}\thanks{Supported by the Swiss National Science Foundation via grant number 200021-156420.}%
\and Thorsten Kleinjung\inst{2}\thanks{This work was mostly done while the author was with the Laboratory for Cryptologic Algorithms, EPFL, Switzerland, supported by the Swiss National Science Foundation via grant number 200020-132160.}%
\and Jens Zumbr\"agel\inst{1}\thanks{This work was mostly done while the author was with the Institute of Algebra, TU Dresden, Germany, supported by the Irish Research Council via grant number ELEVATEPD/2013/82.}}

\institute{Laboratory for Cryptologic Algorithms\\
  School of Computer and Communication Sciences\\
  \'Ecole polytechnique f\'ed\'erale de Lausanne, Switzerland\\
  \and
  Institute of Mathematics, Universit\"at Leipzig, Germany\\
  \email{\{robert.granger,thorsten.kleinjung,jens.zumbragel\}@epfl.ch}}

\maketitle

\begin{abstract}
For~$q$ a prime power, the discrete logarithm problem (DLP) in~$\F_{q}$ consists in finding, for any $g \in \F_{q}^{\times}$ and $h \in \langle g \rangle$, an integer~$x$ such that $g^x = h$.
We present an algorithm for computing discrete logarithms with which we prove that for each prime~$p$ there exist infinitely many explicit extension fields~$\F_{p^n}$ in which the DLP can be solved in expected quasi-polynomial time.
Furthermore, subject to a conjecture on the existence of irreducible polynomials of a certain form, the algorithm solves the DLP in all extensions~$\F_{p^n}$ in expected quasi-polynomial time.
%\keywords{discrete logarithm problem, finite fields, quasi-polynomial time algorithm}
% \subclass{11Y16, 11T71}
\end{abstract}

%-----------------------------------------------------------------------------

\section{Introduction}\label{sec:intro}

In this paper we prove the following result.

\begin{theorem}\label{thm-sequence}
For every prime $p$ there exist infinitely many explicit extension fields $\F_{p^n}$ in which the DLP can be solved in expected quasi-polynomial time 
\begin{equation}\label{thm1:complexity}
\exp \big( ( 1 / \log 2 + o(1) ) (\log n)^2 \big) .
\end{equation}
\end{theorem}

Theorem~\ref{thm-sequence} is an easy corollary of the following much stronger result, which we prove by presenting a randomised algorithm for solving any such DLP.

\begin{theorem}\label{thm-main}
Given a prime power $q>61$ that is not a power of~$4$, an integer~${k \ge 18}$, coprime polynomials $h_0, h_1 \in \F_{q^k}[X]$ of degree at most two and an irreducible degree~$l$ factor~$I$ of ${h_1 X^q - h_0}$, the DLP in $\F_{q^{kl}} \cong \F_{q^k}[X]/(I)$ can be solved in expected time
\begin{equation}\label{thm2:complexity}
 q^{\log_2 l + O( k )}.
\end{equation} 
\end{theorem}

To deduce Theorem~\ref{thm-sequence} from Theorem~\ref{thm-main}, note that thanks to Kummer theory, when $l = q - 1$ such $h_0, h_1$ are known to exist; indeed, for all~$k$ there exists an $a \in \F_{q^k}$ such that $I = X^{q-1} - a \in \F_{q^k}[X]$ is irreducible and therefore $I \mid X^q - a X$. By setting $q = p^i > 61$ for any $i \ge 1$ (odd for $p = 2$),  $k = 18$, $l = q - 1 = p^i - 1$ and finally $n = i k (p^i - 1)$, applying~(\ref{thm2:complexity}) proves that the DLP in this representation of $\F_{p^n}$ can be solved in expected time~(\ref{thm1:complexity}). As one can compute an isomorphism between any two representations of $\F_{p^n}$ in polynomial time~\cite{LenstraIso}, this completes the proof.
Observe that one may replace the prime~$p$ in Theorem~\ref{thm-sequence} by a (fixed) prime power $p^r$ by setting $k = 18 r$ in the argument above.

In order to apply Theorem~\ref{thm-main} to the DLP in $\F_{p^n}$ with~$p$ fixed and arbitrary~$n$, one should first embed the DLP into one in an appropriately chosen $\F_{q^{kn}}$. By this we mean that $q = p^i$ should be at least $n-2$ (so that $h_0,h_1$ may exist) but not too large, and that $18 \le k = o(\log q)$, so that the resulting 
complexity~(\ref{thm2:complexity}) is given by~(\ref{thm1:complexity}) as $n \rightarrow \infty$. Proving that appropriate
$h_0,h_1 \in \F_{q^k}[X]$ exist for such $q$ and $k$ would complete our approach and prove the far stronger result that the DLP in~$\F_{p^n}$ with~$p$ fixed can be solved in expected time~(\ref{thm1:complexity}) for all $n$.
However, this seems to be a very hard problem, even if heuristically it would appear to be almost certain.  

Note that if one could prove the existence of an infinite sequence of primes $p$ (or more generally prime powers) for which $p-1$ is quasi-polynomially smooth in $\log p$, then the Pohlig-Hellman algorithm~\cite{PohligHellman} would also give a rigorous -- and deterministic -- quasi-polynomial time algorithm for solving the DLP in such fields, akin to Theorem~\ref{thm-sequence}. 
However, such a sequence is not known to exist and even if it were, Theorem~\ref{thm-sequence} is arguably more interesting since the present algorithm exploits properties of the fields in question rather than just the factorisation of the order of their multiplicative groups. 
Furthermore, the fields to which the algorithm applies are explicit, whereas it may be very hard to find members of such a sequence of primes (or prime powers), should one exist.

The first (heuristic) quasi-polynomial algorithm for discrete logarithms in finite fields of fixed characteristic was devised by Barbulescu, Gaudry, Joux and Thom\'e~\cite{BGJT}, building upon an approach of Joux~\cite{Joux13a}.
We emphasise that the quasi-polynomial algorithm presented here relies on a different principal building block, whose roots may be found in the work of G\"olo\u{g}lu, Granger,
McGuire and Zumbr\"agel~\cite{GGMZ13a}.
In contrast to the algorithm of Barbulescu \emph{et al.}, the present algorithm eliminates the need for smoothness heuristics; this feature as well as
the algebraic nature of the algorithm makes a rigorous analysis possible.

The sequel is organised as follows.  In Section~\ref{sec:algorithm} we present the algorithm, which involves the repeated application of what is referred to as a descent.  In Section~\ref{sec:descent} we describe our descent method, provide details of its building block and explain why its successful application implies Theorem~\ref{thm-main}, and hence Theorem~\ref{thm-sequence}.  Finally, in Section~\ref{sec:correctness} we complete the proof of these theorems by demonstrating that every step of each descent is successful.

%-----------------------------------------------------------------------------

\section{The algorithm}\label{sec:algorithm}

As per Theorem~\ref{thm-main}, let $q>61$ be a prime power that is not a power of~$4$ and let $k \ge 18$ be an integer; the reasons for these bounds are explained in Sections~\ref{sec:descent} and~\ref{sec:correctness}.  We also assume there exist $h_0, h_1, I \in \F_{q^k}[X]$ satisfying the conditions of Theorem~\ref{thm-main}.
Finally, let $g \in \smash{\F_{q^{kl}}^{\times}}$ and let $h \in \langle g \rangle$ be the target element for the DLP to base $g$.

The structure and analysis of the algorithm closely follows the approach of Diem in the context of the elliptic curve DLP~\cite{Diem},
which is based on that of Enge and Gaudry~\cite{EngeGaudry}.  However, a difference is that it obviates the need to factorise the group order. \medskip

{\sf\noindent 
Input: A prime power $q>61$ that is not a power of~$4$; an integer $k \ge 18$; a positive integer $l$; polynomials $h_0, h_1, I \in \F_{q^k}[X]$ with $h_0, h_1$ being coprime, $\deg(h_0), \deg(h_1) \le 2$ and $I$ a degree $l$ irreducible factor of $h_1 X^q - h_0$; $g \in \F_{q^{kl}}^{\times}$ and $h \in \langle g \rangle$. \medskip

\noindent Output: An integer $x$ such that $g^x = h$.

\begin{enumerate}[\quad 1.]
\item Let $N = q^{kl}-1$, let $\mathcal{F} = \{ F \in \Fqk[X] \mid \deg F \le 1,\, F \ne 0 \} \cup \{ h_1 \}$ and denote its elements by $F_1, \ldots, F_m$, where $m = |\mathcal{F}| = q^{2k}$ (or $q^{2k} - 1$ if $\deg h_1 \le 1$).
\item\label{step:construct} Construct a matrix $R = (r_{i,j}) \in (\ZNZ)^{(m+1) \times m}$ and column vectors $\alpha,\beta \in (\ZNZ)^{m+1}$ as follows.
For each $i$ with $1 \le i \le m+1$ choose $\alpha_i, \beta_i\in \ZNZ$ uniformly and independently at random and apply the (randomised) descent algorithm of Section~\ref{sec:descent} to $g^{\alpha_i}h^{\beta_i}$ to express this as\vspace{-2mm}
\[ g^{\alpha_i} h^{\beta_i} = \prod_{j=1}^m ( F_j \bmod I )^{r_{i,j}} . \vspace{-2mm} \]
\item Compute a lower row echelon form $R'$ of $R$ by using invertible row transformations; apply these row transformations also to $\alpha$ and $\beta$, and denote the results by $\alpha'$ and $\beta'$.
\item If $\gcd(\beta_1',N)>1$, go to Step~\ref{step:construct}.
\item Return an integer $x$ such that $\alpha_1' + x \beta_1' \equiv 0 \pmod{N}$.
\end{enumerate}} \smallskip

We now explain why the algorithm is correct and discuss the running time, treating the descent in Step 2 as a black box algorithm for now.  Henceforth, we assume that any random choices used in the descent executions are independent from each other and of the randomness of~$\alpha$ and~$\beta$.
For the correctness, note that~$g^{\alpha_1'} h^{\beta_1'} = 1$ holds after Step 3, since the first row of~$R'$ vanishes.  Thus for any integer~$x$ such that $\alpha_1' + x \beta_1' \equiv 0 \pmod{N}$ we have $g^x = h$, provided that~$\beta_1'$ is invertible in $\Z / N \Z$.

\begin{lemma}
  After Step~3 of the algorithm the element $\beta_1' \in \Z / N \Z$ is uniformly distributed.  Therefore, the algorithm succeeds with probability $\varphi(N) / N$, where $\varphi$ denotes Euler's phi function.
\end{lemma}

\begin{proof}
  We follow the argument from~\cite[Sec.~5]{EngeGaudry} and~\cite[Sec.~2.3]{Diem}.  As $h \in \langle g \rangle$, for any fixed value $\beta_i = b \in \Z / N \Z$ the element $g^{\alpha_i} h^b$ is uniformly distributed over the group $\langle g \rangle$, therefore the element $g^{\alpha_i} h^{\beta_i}$ is independent of~$\beta_i$.
  As the executions of the descent algorithm are assumed to be independent, we have that the row $(r_{i,1}, \ldots, r_{i,m})$ is also independent of~$\beta_i$.
  It follows that the matrix~$R$ is independent of the vector~$\beta$.  Then the (invertible) transformation matrix $U \in (\Z / N \Z)^{(m+1) \times (m+1)}$ is also independent of~$\beta$, so that $\beta' = U \beta$ is uniformly distributed over $(\Z / N \Z)^{m+1}$, since~$\beta$ is.  From this the lemma follows.
\end{proof}

Regarding the running time, for Step~3 we note that a lower row echelon form of~$R$ can be obtained using invertible row transformations as for the Smith normal form, which along with the corresponding transformation matrices can be computed in polynomial time~\cite{SmithNF}, so that Step~3 takes time polynomial in~$m$ and $\log N$.
Furthermore, from~\cite{EulerPhi} we obtain $N / \varphi(N) \in O(\log \log N)$.
Altogether this implies that the DLP algorithm has quasi-polynomial expected running time (in~$\log N$), provided the descent is quasi-polynomial.  We defer a detailed complexity analysis of the descent to Section~\ref{sec:descent}.

Observe that the algorithm does not require $g$ to be a generator of~$\F_{q^{kl}}^{\times}$, which is in practice hard to test without factorising $N$.
In fact, the algorithm gives rise to a Monte Carlo method for deciding group membership $h \in \langle g \rangle$.  Indeed, if a discrete logarithm $\log_g h$ has been computed, then obviously $h \in \langle g \rangle$; thus if $h \not\in \langle g \rangle$, we always must have $\gcd(\beta_1', N) > 1$ in Step~4.

Practitioners may have noticed inefficiencies in the algorithm.  For example, in the usual index calculus method one precomputes the logarithms of all factor base elements and then applies a single descent to the target element to obtain its logarithm.
Moreover, one usually first computes the logarithm in $\F_{q^{kl}}^{\times} / \F_{q^k}^{\times}$, i.e., one ignores multiplicative constants and therefore includes only monic polynomials in the factor base, obtaining the remaining information by solving an additional DLP in $\F_{q^k}^{\times}$.
However, the setup as presented simplifies and facilitates our rigorous analysis.

%-----------------------------------------------------------------------------

\section{The descent}\label{sec:descent}

In this section we detail the building block behind our descent method and explain why its successful application implies Theorem~\ref{thm-main}.
Let~$q$ be a prime power, $k$ and~$l$ positive integers and let $R = \F_{q^k}[X,Y]$.
The setup for the target field~$\F_{q^{kl}}$ has irreducible polynomials $f_1 = Y - X^q \in R$ and $f_2 = h_1 Y - h_0 \in R$ with $h_0, h_1 \in \F_{q^k}[X]$ coprime of degree at most two and $h_1 X^q - h_0$ having an irreducible factor~$I$ of degree $l$, i.e., $R_{12} = \F_{q^k}[X,Y]/(f_1,f_2)$ is a finite ring surjecting onto~$\F_{q^{kl}} = \F_{q^k}[X] / (I)$.%
\footnote{One can equally well work with $f_2 = h_1 X - h_0$ with $h_i \in \F_{q^k}[Y]$ of degree at most two, where $h_1(X^q) X - h_0(X^q)$ has a degree~$l$ irreducible factor, as proposed in~\cite{GKZ14a}, with all subsequent arguments holding {\em mutatis mutandis}.}
This implies $R_1 = R / (f_1) \cong \F_{q^k}[X]$ and $R_2 = R / (f_2) \cong \F_{q^k}[X][\frac1{h_1}]$,
and from now on we identify elements in~$R_1$ and~$R_2$ with expressions in~$X$ via these isomorphisms.
The setup is summarised in Fig.~\ref{FFSpic}.
\begin{figure}[h]
  \begin{center}
    \begin{tikzpicture}
      \matrix (m) [matrix of math nodes, row sep = 3em, column sep = -1em] {
        & R = \F_{q^k}[X,Y] & \\
        R_1 = R / (f_1) & & R_2 = R / (f_2) \\
        & R_{12} = R / (f_1, f_2) & \\
        & \F_{q^{kl}} & \\
      };
      \path[-stealth] 
        (m-1-2) edge (m-2-1) edge (m-2-3)
        (m-2-1) edge (m-3-2)
        (m-2-3) edge (m-3-2)
        (m-3-2) edge (m-4-2);
    \end{tikzpicture}
  \end{center}
  \caption{Setup for the target field $\F_{q^{kl}}$}\label{FFSpic}
\end{figure}
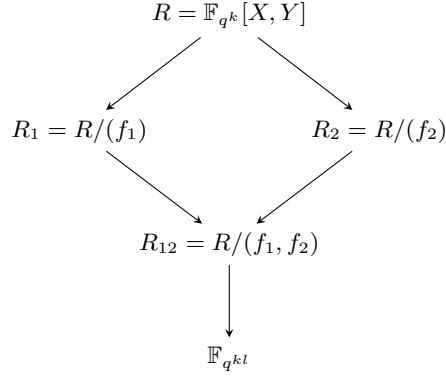

By the phrase ``rewriting a polynomial $Q$ (in $R_1$ or $R_2$) in terms of polynomials~$P_i$ (in~$R_1$ or~$R_2$)'' we henceforth mean that in the target field the image of~$Q$ equals a product of (positive or negative) powers of images of~$P_i$.
If the~$P_i$ are of lower degree then one has \emph{eliminated} the polynomial~$Q$.
Typically such rewritings are obtained by considering $\mathcal{P} \bmod f_1 \in R_1$ and $\mathcal{P} \bmod f_2 \in R_2$, where $\mathcal{P} \in R$.
Since~$h_1$ usually appears in $\mathcal{P} \bmod f_2$, it is adjoined to the factor base~$\mathcal{F}$, and for the sake of simplicity it is sometimes suppressed in the following description.
Accordingly, a \emph{descent} is an algorithm that rewrites any given nonzero target field element, represented by a polynomial~$Q$, in terms of polynomials~$F_j$ of the factor base, i.e., of degree $\le 1$.

\subsection{Degree two elimination}

In this subsection we review the on-the-fly degree two elimination method from~\cite{GGMZ13a}, adjusted for the present framework. In~\cite{bluher} the major portion of the set of polynomials obtained as linear fractional transformations of $X^q-X$ is parameterised as follows.
Let $\mathcal{B}_k$ be the set of $B \in \smash{\F_{q^k}^{\times}}$ such that the polynomial $X^{q+1}-BX+B$ splits completely over $\smash{\F_{q^k}}$, the cardinality of which is approximately $q^{k-3}$~\cite[Lemma 4.4]{bluher}.
Scaling and translating these polynomials means that all the polynomials $X^{q+1}+aX^q+bX+c$ with $c \ne ab$, $b \ne a^q$ and $B = \smash{\frac{(b -a^q)^{q+1}}{(c - ab)^q}}$ split completely over $\F_{q^k}$ whenever $B \in \mathcal{B}_k$.

Let $Q$ (viewed as a polynomial in $R_2$) be an irreducible quadratic polynomial to be eliminated.  We let $L_Q \subset \F_{q^k}[X]^2$ be the lattice defined by
\begin{equation}\label{eq:lq}
  L_Q = \{ (w_0, w_1) \in \F_{q^k}[X]^2 \mid w_0 h_0 + w_1 h_1 \equiv 0 \!\!\pmod{Q} \}.
\end{equation}
In the case that $Q$ divides $w_0 h_0 + w_1 h_1 \ne 0$ for some $w_0, w_1 \in \F_{q^k}$, then $Q = w (w_0 h_0 + w_1 h_1)$ for some $w \in \smash{\F_{q^k}^{\times}}$, since the degree on the right hand side is at most two.  
Therefore, $Q$ can be rewritten in terms of $w_0 X^q + w_1 = \smash{(w_0^{1/q} X + w_1^{1/q})^q} \in R_1$ (and~$h_1$), by considering the element $\mathcal{P} = w_0 Y + w_1 \in R$.
We will say in this case that the lattice is degenerate.

In the other (non-degenerate) case, $L_Q$ has a basis of the form $(1, {u_0 X + u_1})$, $(X, {v_0 X + v_1})$ with $u_i,v_i \in \F_{q^k}$.
Since the polynomial $\mathcal{P} = XY + aY + bX + c$ maps to $\frac 1 {h_1} ((X+a)h_0 + (bX+c)h_1)$ in~$R_2$, $Q$~divides~$\mathcal{P} \bmod f_2$ if and only if $(X+a, bX+c) \in L_Q$.
Note that the numerator of $\mathcal{P} \bmod f_2$ is of degree at most three, thus it can at worst contain a linear factor besides~$Q$.
If the triple $(a,b,c)$ also satisfies $c \ne ab$, $b \ne a^q$ and $\smash{\frac{(b - a^q)^{q+1}}{(c - ab)^q} \in \mathcal{B}_k}$, then $\mathcal{P} \bmod f_1$ splits into linear factors and thus~$Q$ has been rewritten in terms of linear polynomials.

Algorithmically, a triple $(a,b,c)$ satisfying all conditions can be found in several ways.
Choosing a $B \in \mathcal{B}_k$, considering $(X+a, bX+c) = a(1, u_0 X \!+ u_1) + (X, v_0 X \!+ v_1)$ and rewriting $b = u_0 a + v_0$ and $c = u_1 a + v_1$ gives the condition
\begin{equation}\label{eq:deg2elim}
  B = \frac{(-a^q + u_0 a +v_0)^{q+1}}{(-u_0 a^2 + (-v_0 + u_1)a + v_1)^q}.
\end{equation}
By expressing $a$ in an $\F_{q^k}/\F_q$ basis, (\ref{eq:deg2elim}) results in a quadratic system in $k$ variables~\cite{GGMZ13b}. Using a Gr\"obner basis algorithm the running time is exponential in~$k$.
Alternatively, and this is one of the key observations for the present work, equation~(\ref{eq:deg2elim}) can be considered as a polynomial of degree $q^2+q$ in $a$ whose roots can be found in (deterministic) polynomial time in~$q$ and in~$k$ by using an algorithm of Berlekamp~\cite{Berlekamp70}.
One can also check for random $(a,b,c)$ such that the lattice condition holds, whether $X^{q+1} + aX^q + bX + c$ splits into linear polynomials, which happens with probability~$q^{-3}$. Each such instance is also polynomial time in~$q$ and in~$k$.

These degree~$2$ elimination methods will fail when~$Q$ divides $h_1 X^q - h_0$, because this would imply that the polynomial $\mathcal{P} \bmod f_1 = X^{q+1} + aX^q + bX + c$ is divisible by~$Q$ whenever $\mathcal{P} \bmod f_2$ is, a problem first discussed in~\cite{traps}.
Such polynomials~$Q$ or their roots will be called traps of level~$0$.  Similarly, these degree $2$ elimination methods might also fail when $Q$ divides $h_1 X^{q^{k+1}} - h_0$, in which case such polynomials~$Q$ or their roots will be called traps of level~$k$. 

Note that for Kummer extensions, i.e., when $h_1 = 1$ and $h_0 = a X$ for some $a \in \F_{q^k}$, there are no traps and hence much of the following treatment is not required for proving only Theorem~\ref{thm-sequence}.  
However, it is essential to consider traps for proving the far more general Theorem~\ref{thm-main}.

\subsection{Elimination requirements}

\begin{sloppypar} The degree two elimination method can be transformed into an elimination method for irreducible even degree polynomials.  We now present a theorem which states that under some assumptions this degree two elimination is guaranteed to succeed, and subsequently demonstrate that it implies Theorem~\ref{thm-main}. \end{sloppypar}

An element $\tau \in \overline \F_{q^k}$ for which $[\F_{q^k}(\tau):\F_{q^k}]
= 2d$ is even and $h_1(\tau) \ne 0$, is called a {\em trap root} if it is a 
root of $h_1 X^q - h_0$ or $\smash{h_1 X^{q^{kd+1}} - h_0}$, or if 
$\frac{h_0}{h_1}(\tau) \in \F_{q^{kd}}$.
Note that the sets of trap roots is invariant under the absolute
Galois group of~$\F_{q^k}$.
A polynomial in~$R_1$ or~$R_2$ is said to be {\em good} if it has no
trap roots; the same definitions are used when the base field of~$R_1$
and~$R_2$ is extended.
This definition encompasses traps of level~$0$, of level~$kd$,
and the case where for $Q \ne h_1$ the lattice~$L_Q$ is degenerate.

\begin{theorem}\label{thm-onthefly}
  Let $q > 61$ be a prime power that is not a power of $4$, let $k \ge 18$ be an integer and let $h_0, h_1 \in \F_{q^k}[X]$ be coprime polynomials of degree at most two with $h_1 X^q - h_0$ having an irreducible degree~$l$ factor. 
  Moreover, let $d \ge 1$ be an integer, let $Q \in \F_{q^{kd}}[X]$, $Q \ne h_1$ be an irreducible quadratic good polynomial, and let $(1, {u_0 X + u_1}), (X, {v_0 X + v_1})$ be a basis of the lattice $L_Q$ in (\ref{eq:lq}), now over~$\F_{q^{kd}}$. 
  Then the number of solutions $(a,B) \in \F_{q^{kd}} \times \mathcal{B}_{kd}$ of (\ref{eq:deg2elim}) resulting in good descendents is at least~$q^{kd-5}$.
\end{theorem}

This theorem is of central importance for our rigorous analysis and is proven in Section~\ref{sec:correctness}.

\subsection{Degree $2d$ elimination and descent complexity}

Now we demonstrate how the degree two elimination gives rise to a method for eliminating irreducible even degree polynomials, which is the crucial building block for our descent algorithm.
As per Theorem~\ref{thm-onthefly}, let $q > 61$ be a prime power that is not a power of $4$, let $k \ge 18$, and let $h_0, h_1, I$ as before.

\begin{proposition}\label{prop-2d}
  Let $d \ge 1$ and $Q \in R_2$, $Q \ne h_1$, be an irreducible good polynomial of degree~$2d$.  Then~$Q$ can be expressed in terms of at most $q+2$ irreducible good polynomials of degrees dividing~$d$, in an expected running time polynomial in~$q$ and in~$d$.
\end{proposition}

\begin{proof}
  Over the extension $\F_{q^{kd}}$ the polynomial $Q$ splits into $d$ irreducible good quadratic polynomials, which are all conjugates under $\Gal(\F_{q^{kd}} / \F_{q^k})$; let~$Q'$ be one of them.  Since $Q' \ne h_1$ is good it does not divide $w_0 h_0 + w_1 h_1 \ne 0$ for some $w_0, w_1 \in \F_{q^{kd}}$.
  By Theorem~\ref{thm-onthefly}, with an expected polynomial number of trials, the degree two elimination method for $Q' \in \F_{q^{kd}}[X]$ produces a polynomial $P' \in \F_{q^{kd}}[X,Y]$ such that $P' \bmod f_1$ splits into a product of at most $q+1$ good polynomials of degree one over $\F_{q^{kd}}$ and such that $(P' \bmod f_2)h_1$ is a product of $Q'$ and a good polynomial of degree at most one.
  Let $P$ be the product of all conjugates of $P'$ under $\Gal(\F_{q^{kd}} / \F_{q^k})$.  As the product of all conjugates of a linear polynomial under $\Gal(\F_{q^{kd}} / \F_{q^k})$ is the $d_1$-th power of an irreducible degree~$d_2$ polynomial for~$d_1$ and~$d_2$ satisfying $d_1 d_2 = d$, the rewriting assertion of the proposition follows.

  The three steps of this method -- computing $Q'$, the degree two elimination (when the second or third approach listed above for solving (\ref{eq:deg2elim}) is used), and the computation of the polynomial norms -- all have running time polynomial in $q$ and in $d$, which proves the running time assertion.
\end{proof}

By recursively applying Proposition~\ref{prop-2d} we can express a good irreducible polynomial of degree $2^e$, $e \ge 1$, in terms of at most $(q+2)^{e}$ linear polynomials.
The final step of this recursion, namely eliminating up to $(q+2)^{e-1}$ quadratic polynomials, dominates the running time, which is thus upper bounded by $(q+2)^e$ times a polynomial in~$q$.

\begin{lemma}\label{lem-wan}
  Any nonzero element in $\F_{q^{kl}}$ can be lifted to an irreducible good polynomial of degree~$2^e$ in $\F_{q^k}[X]$, provided that $2^e > 4l$.
\end{lemma}

\begin{proof}
  By the effective Dirichlet-type theorem on irreducibles in arithmetic progressions \cite[Thm.~5.1]{Wan}, for $2^e > 4l$ the probability of irreducibility for a random lift is lower bounded by $2^{-e-1}$.  One may actually find an irreducible polynomial of degree $2^e$ which is good, since the number of possible trap roots ($<q^{k2^{e-1}+2}$) is much smaller than the number ($>q^{k(2^e-l)}2^{-e-1}$) of irreducibles produced by this Dirichlet-type theorem.
\end{proof}

Finally, putting everything together (and assuming Theorem~\ref{thm-onthefly}) proves the quasi-polynomial expected running time of a descent and therefore the running time of the algorithm, establishing Theorem~\ref{thm-main}.

Note that when $q = L_{q^{kl}} (\alpha)$, where $L_N(\alpha)$ for $\alpha \in [0, 1]$ is the usual subexponential function $\exp( O( (\log N)^{\alpha} (\log \log N)^{1 - \alpha} ) )$,
as in~\cite{BGJT} the complexity stated in Theorem~\ref{thm-main} is $L_{q^{kl}} (\alpha + o(1))$, which is therefore better than the classical function field sieve for 
$\alpha < \frac 1 3$.

\enlargethispage{5mm}

Also note that during an elimination step, one need not use the basic building block as stated, which takes the norms of the linear polynomials produced back down to $\F_{q^{k}}$. Instead, one need only take their norms to a subfield of index $2$, thus becoming quadratic polynomials, and then recurse, as depicted in Fig.~\ref{descentpic}.

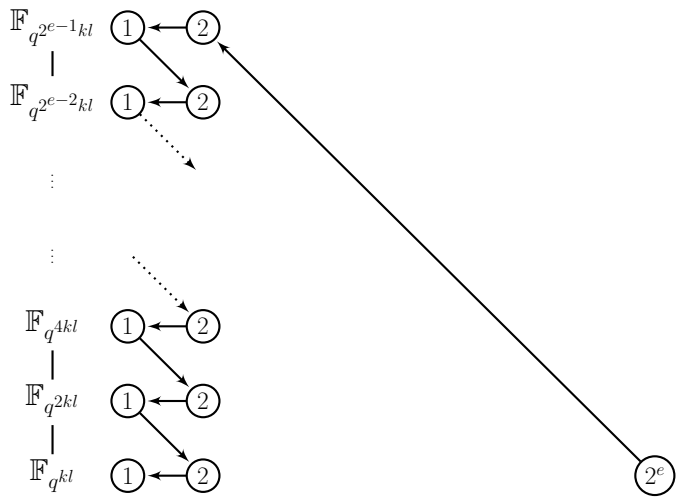
\begin{figure}[h]
\centering
\tikzstyle{line} = [draw, -latex']
\begin{center}
\begin{tikzpicture}[->,>=stealth',shorten >=1pt,auto,node distance=1.8cm,
  thick,main node/.style={circle,draw,font=\sffamily\LARGE\bfseries}, scale=0.55, transform shape]

\node[main node] (1) {$1$};
\node[main node] (2) [right of=1] {$2$};
\node (4) [right of=2] {};
\node (8) [right of=4] {};
\node (16) [right of=8] {};
\node (dots) [right of=16] {};
\node (dots2) [right of=dots] {};
\node[main node] (2_e) [right of=dots2] {\LARGE{$2^e$}};
\node (fqk) [left of=1] {\huge{$\F_{q^{kl}}$}};
\path [line] (2) edge (1);

\node (fq2k)[above of=fqk]{\huge{$\F_{q^{2kl}}$}};
\node[main node] (1a) [right of=fq2k] {$1$};
\node[main node] (2a) [right of=1a] {$2$};
\path [line] (2a) edge (1a);
\path [line] (1a) edge (2);
\path [line] (fqk) edge [-] (fq2k);

\node (fq4k)[above of=fq2k]{\huge{$\F_{q^{4kl}}$}};
\node[main node] (1b) [right of=fq4k] {$1$};
\node[main node] (2b) [right of=1b] {$2$};
\path [line] (fq2k) edge [-] (fq4k);
\path [line] (2b) edge (1b);
\path [line] (1b) edge (2a);

\node (fq8k)[above of=fq4k]{$\vdots$};
\node (1c) [right of=fq8k] {};
\node (2c) [right of=1c] {};
\path [line,dotted] (1c) edge node {} (2b);

\node (vdots) [above of=fq8k] {$\vdots$};
\node (1d) [right of=vdots] {};
\node (2d) [right of=1d] {};

\node (fq2e2k) [above of=vdots] {\huge{$\F_{q^{2^{e-2}kl}}$}};
\node[main node] (1e) [right of=fq2e2k] {$1$};
\node[main node] (2e) [right of=1e] {$2$};
\path [line] (2e) edge (1e);
\path [line,dotted] (1e) edge node {} (2d);

\node (fq2e1k)[above of=fq2e2k]{\huge{$\F_{q^{2^{e-1}kl}}$}};
\node[main node] (1f) [right of=fq2e1k] {$1$};
\node[main node] (2f) [right of=1f] {$2$};
\path [line] (fq2e1k) edge [-] (fq2e2k);
\path [line] (2f) edge (1f);
\path [line] (1f) edge (2e);

\path [line] (2_e) edge (2f);

\end{tikzpicture}
\end{center}
\caption{Elimination of irreducible polynomials of degree a power of $2$ when considered as elements of $\F_{q^k}[X]$. The arrow directions $\nwarrow, \leftarrow$ and 
$\searrow$ indicate factorisation, degree 2 elimination and taking a norm with respect to the indicated subfield, respectively.
(We have suppressed the rare cases, where linear polynomials are already in a subfield of index 2.)}\label{descentpic}
\end{figure}

%-----------------------------------------------------------------------------

\section{Proof of Theorem~\ref{thm-onthefly}}\label{sec:correctness}

In this section we prove Theorem~\ref{thm-onthefly}, which by the arguments of the previous section demonstrates the correctness of the algorithm and the main theorems.

\subsection{Notation and statement of supporting results}\label{A1}

Let $K = \F_{q^{kd}}$ where ${kd \ge 18}$, let $L = \F_{q^{2kd}}$ be its quadratic extension, and 
let $\mathcal{B}$ be the set of $B \in K^{\times}$ such that the polynomial $X^{q+1} - BX + B$ splits completely over~$K$.  
Using an elementary extension of~\cite[Prop.~5]{HellesethKholosha} we have the following characterisation; we add a short proof for the reader's convenience.

\begin{lemma}\label{lem:bchar}
  The set~$\mathcal{B}$ equals the image of $K \setminus \F_{q^2}$
  under the map
  \[ u \mapsto \frac{(u - u^{q^2})^{q+1}} {(u - u^q)^{q^2+1}} . \]
\end{lemma}

\begin{proof}
  We consider the action of $\PGL_2(K)$ on polynomials, cf.\
  Subsection~\ref{ssec:pgl}.  For $u \in K \setminus \F_{q^2}$ the 
  matrix 
  \[ \begin{pmatrix} {\lambda} & {0} \\ {0} & {1} \end{pmatrix}
  \begin{pmatrix} {1} & {\mu} \\ {0} & {1} \end{pmatrix}
  \begin{pmatrix} {1} & {0} \\ {u} & {1} \end{pmatrix}
  \text{ with }
  \lambda = \frac{(u-u^q)^q} {(u-u^q) (u-u^{q^2})}
  \text{ and }
  \mu = -\frac 1 {u-u^q} \] transforms the polynomial
  $X^q - X$ into $X^{q+1} - B X + B$ with
  $B = \smash{\frac{(u-u^{q^2})^{q+1}} {(u-u^q)^{q^2+1}}}$.
  Thus the set $\mathcal{B}$ contains the image of the map.
  
  Conversely, assume that $X^{q+1} - B X + B$ splits completely and $B
  \ne 0$.  Since the polynomial has no double roots, it is $X^q - X$
  transformed under some $g \in \PGL_2(K)$.  As the polynomial has
  degree $q+1$ the matrix $g$ can be decomposed as above, a priori
  with different $\lambda$ and $\mu$.  Since the shape of the
  polynomial determines~$\lambda$ and~$\mu$ in terms of $u$, $B$ must
  be as above.
\end{proof}

Now let~$Q$ be an irreducible quadratic polynomial in $K[X]$ such that
a basis of its associated lattice $L_Q$ in~(\ref{eq:lq}), now over~$K$,
is given by $(1, {u_0 X + u_1}), (X, {v_0 X + v_1})$.  Then $Q$ is a
scalar multiple of~${-u_0 X^2 + (-u_1 + v_0) X + v_1}$.
By Lemma~\ref{lem:bchar} and~(\ref{eq:deg2elim}), in order to eliminate $Q$ we need to find $(a,u) \in K \times (K \setminus \Ft)$ satisfying
\[ (u - u^{q^2})^{q+1}(-u_0 a^2 + (-v_0 + u_1)a + v_1)^q -  (u - u^q)^{q^2+1} (-a^q + u_0 a +v_0)^{q+1}   = 0. \]
The two terms have a common factor $(u-u^q)^{q+1}$ which motivates the following definitions.
Let $\alpha=-u_0$, $\beta=u_1-v_0$, $\gamma=v_1$ and $\delta=-v_0$ with $\alpha,\beta,\gamma,\delta \in K$, as well as
\begin{gather*}
  D = \frac{U^{q^2}-U}{U^q-U} = \prod_{\epsilon \in \Ft \setminus \Fo} 
  (U-\epsilon), \\
  E = U^q-U=\prod_{\epsilon \in \Fo} (U-\epsilon), \\
  F = \alpha A^2 + \beta A + \gamma =  \alpha (A-\rho_1) (A-\rho_2) \
  \text{ with } \ \rho_1,\rho_2 \in L, \\
  G = A^q+\alpha A+\delta \quad \text{and} \\
  P = D^{q+1}F^q-E^{q^2-q}G^{q+1} \in K[A,U]. 
\end{gather*}
Note that $F$ equals $Q(-A)$ (up to a scalar), so that $\deg(F) = 2$, $F$ is irreducible and $\rho_1, \rho_2 \notin K$.
We consider the curve~$C$ defined by $P = 0$ and are interested in the number of (affine) points $(a,u) \in C(K)$ with $u \notin \Ft$.
More precisely, we want to prove the following.

\begin{theorem}\label{thm:heu2}
  Let $q>61$ be a prime power that is not a power of $4$.  If the conditions%
  \vspace{-3mm}
  \begin{gather*}
    (*) \qquad \rho_1^q+\alpha\rho_2+\delta \ne 0 \\
    (**) \qquad \rho_1^q+\alpha\rho_1+\delta \ne 0
  \end{gather*}
  hold then there are at least $q^{kd-1}$ pairs $(a,u) \in K \times (K \setminus \Ft)$ satisfying $P(a,u)=0$.
\end{theorem}

The relation of the two conditions to the quadratic polynomial $Q$ as well as properties of traps are described in the following propositions.

\begin{proposition}\label{prop:cond}
  If condition $(*)$ is not satisfied, then $Q$ divides $h_1 X^q - h_0$, i.e., $Q$~is a trap of level~$0$.
  If  condition $(**)$ is not satisfied, then $Q$ divides $\smash{h_1 X^{q^{kd+1}} - h_0}$, i.e., $Q$ is a trap of level~$kd$.
  In particular, if~$Q$ is a good polynomial then conditions~$(*)$ and~$(**)$ are satisfied.
\end{proposition}

\begin{proposition}\label{prop:trap}
  Let $(a,u), (a',u') \in K \times (K \setminus \Ft)$ be two solutions of $P=0$ with $a \ne a'$, corresponding to the polynomials ${\mathcal P}_a=XY+aY+bX+c$ and ${\mathcal P}_{a'}=XY+a'Y+b'X+c'$, respectively.
  Then ${\mathcal P}_a \bmod f_1$ and ${\mathcal P}_{a'} \bmod f_1$ have no common roots.  Furthermore, the common roots of ${\mathcal P}_a \bmod f_2$ and ${\mathcal P}_{a'} \bmod f_2$ are precisely the roots of~$Q$.
\end{proposition}

Now we explain how (for $q > 61$ not a power of~$4$) Theorem~\ref{thm-onthefly} follows from the above theorem and the propositions.
Since the irreducible quadratic polynomial~$Q$ is good, the lattice~$L_Q$ is non-degenerate so that a basis as above exists, and by Proposition~\ref{prop:cond} the two conditions of Theorem~\ref{thm:heu2} are satisfied.
The map of Lemma~\ref{lem:bchar} is $q^3-q:1$ on $K \setminus \Ft$, hence there are at least $q^{kd-4}$ solutions $(a,B) \in K \times \mathcal{B}$ of~(\ref{eq:deg2elim}), which contain at least $q^{kd-4}$ different values $a \in K$.
Observe that a trap root~$\tau$ that may occur in this situation is a root of $h_1 X^q - h_0$, or of $\smash{h_1 X^{{q^{k d'} + 1}} - h_0}$ for $d' \mid \frac d 2$, or
it satisfies $\frac{h_0} {h_1} (\tau) \in \F_{q^{k d / 2}}$.
The cardinality of~these trap roots is at most $\smash{q^{\frac{kd}2+3}}$.
By Proposition~\ref{prop:trap} a trap root can appear in ${\mathcal P}_a \bmod f_j$ for at most two values $a$, at most once for $j=1$ and at most once for $j=2$.
Hence there are at most $\smash{q^{\frac{kd}2+4}} \le q^{kd-5}$ values $a$ for which a trap root appears in ${\mathcal P}_a \bmod f_j$, $j=1,2$.
Thus there are at least $q^{kd-5}$ different values $a$ for which a solution $(a,B)$ leads to an elimination into good polynomials.
This finishes the proof of Theorem~\ref{thm-onthefly}, hence we focus on proving the theorem and the two propositions above.

\subsection{Outline of the proof method}

The main step of the proof of the theorem consists in showing that, subject to conditions $(*)$ and $(**)$, there exists an absolutely irreducible factor $P_1$ of $P$ that lies already in $K[A,U]$.
Since the (total) degree of $P_1$ is at most $q^3+q$, restricting to the component of the curve defined by $P_1$ and using the Weil bound for possibly singular plane curves gives a lower bound on the cardinality of $C(K)$ which is large enough to prove the theorem after accounting for projective points and points with second coordinate in $\Ft$.
This argument is given in the next subsection before dealing with the more involved main step.

For proving the main step the action of $\PGL_2(\Fo)$ on the variable $U$ is considered.  An absolutely irreducible factor $P_1$ of $P$ is stabilised by a subgroup $S_1 \subset \PGL_2(\Fo)$ satisfying~some conditions.
The first step is to show that, after possibly switching to another absolutely irreducible factor, there are only a few cases for the subgroup.  Then for each case it is shown that the factor is defined over $K[A,U]$ or that one of the conditions on the parameters is not satisfied.

The propositions are proven in the final subsection.

\subsection{Weil bound}

Let~$C_1$ be the absolutely irreducible plane curve defined by $P_1$ of degree~$d_1 \le q^3 + q$.  Corollary~2.5 of~\cite{singweil} shows that
\[ |\#C_1(K)-q^{kd}-1| \le (d_1-1)(d_1-2)q^\frac{kd}2 . \]
Since $\deg_A(P_1) \le q^2+q$ there are at most $q^4 + q^3$ affine points with $u \in \Ft$.
The number of points at infinity is at most $d_1 \le q^3+q < q^4$.
Denoting by $C_1(K){\widetilde{\,\,}}$ the set of affine points in $C_1(K)$ with second coordinate $u \not\in \Ft$ one obtains
\[ |\#C_1(K){\widetilde{\,\,}}| > q^{kd} - (q^4+q^3) - d_1 - (d_1-1) (d_1-2) q^\frac{kd}2 > q^{kd} - q^{\frac{kd}2+8} \ge q^{kd-1} , \]
since $kd \ge 18$, thus proving the theorem if there exists an absolutely irreducible factor~$P_1$ defined over $K[A,U]$.

\subsection{$\PGL_2$ action}\label{ssec:pgl}

Here the following convention for the action of $\PGL_2(\Fo)$ on $\PP^1$ and on polynomials is used.
A matrix $\matt{a}{b}{c}{d} \in \PGL_2(\Fo)$ acts on $\PP^1(M)$, where $M$ is an arbitrary field containing $\F_q$, by \[ (x_0:x_1) \mapsto \matt{a}{b}{c}{d}(x_0:x_1)=(ax_0+bx_1:cx_0+dx_1) \] or, via $\PP^1(M)=M \cup \{\infty\}$, by $x \mapsto \frac{ax+b}{cx+d}$.
This is an action on the left, i.e., for $\sigma,\tau \in \PGL_2(\Fo)$
and $x \in \PP^1(M)$ the following holds:
$\sigma (\tau(x)) = (\sigma \tau)(x)$.
On a homogeneous polynomial $H$ in the variables $(X_0:X_1)$ the action
of $\sigma=\matt{a}{b}{c}{d}$
is given by $H^\sigma(X_0:X_1)=H(aX_0+bX_1:cX_0+dX_1)$.
This is an action on the right, satisfying
$H^{(\sigma\tau)}=(H^\sigma)^\tau$.
In the following we will usually use this action on the dehomogenised
polynomials given by $H^\sigma(X)=H(\frac{aX+b}{cX+d})$, clearing
denominators in the appropriate way.

The polynomial $P \in (K[A])[U]$ is invariant under $\PGL_2(\Fo)$ acting on the variable $U$; this can be seen by considering the actions of $\matt{a}{0}{0}{1}$, $\matt{1}{b}{0}{1}$ and $\matt{0}{1}{1}{0}$, and noticing that $\PGL_2(\Fo)$ is generated by these matrices. 
 Let
\[ P = s\prod_{i=1}^g P_i, \qquad P_i \in (\overline{K}[A])[U], \ s \in \overline{K}[A], \]
be the decomposition of $P$ in $(\overline{K}[A])[U]$ into irreducible factors $P_i$ and possibly reducible $s$.
Notice that $s$ must divide $F^q$ and $G^{q+1}$, hence it divides a power of $\gcd(F,G)$.  As~$F$ is irreducible, $\gcd(F,G)$ is either constant or of degree two.  In the latter case $\rho_1$ is a root of $G$ contradicting condition $(**)$.
Therefore one can assume that $s \in \overline{K}$ is a constant.

Let 
\[ P = F^q \prod_{i=1}^{q^3-q} (U-r_i), \qquad r_i \in \overline{K(A)}, \]
be the decomposition of $P$ in $\overline{K(A)}[U]$.
Then $\PGL_2(\Fo)$ permutes the set $\{r_i\}$ and, since fixed points of $\PGL_2(\Fo)$ lie in $\Ft$ but $r_i \notin \Ft$, the action is free.  Since $\#\PGL_2(\Fo)=q^3-q$ the action is transitive.

Therefore the action on the decomposition over $\overline{K}[A,U]$ is also transitive (adjusting the~$P_i$ by scalars in $\overline{K}[A]$ if necessary).
Denoting by $S_i \subset \PGL_2(\Fo)$ the stabiliser of $P_i$ it follows that all $S_i$ are conjugates of each other, thus they have the same cardinality and hence $q^3-q=g\cdot \#S_i$.
Moreover the degree of~$P_i$ in~$U$ is constant, namely $\deg_U (P_i) = \#S_i$, and also the degree of~$P_i$ in~$A$
is constant, thus $g \mid q^2+q=\deg_A(P)$.  In particular, $q-1 \mid \#S_i$ and $\deg_A (P_i) = \smash{\frac{\# S_i} {q - 1}}$.

\subsection{Subgroups of $\PGL_2$}

The classification of subgroups of $\PSL_2(\Fo)$ is well known~\cite{Dickson} and allows to determine all subgroups of $\PGL_2(\Fo)$~\cite{Cameron}.
Since $\#S_i$ is divisible by $q-1$ (in particular $\#S_i>60$), only the following subgroups are of interest (per conjugation class only one subgroup is listed):
\begin{enumerate}[\quad 1.]
\item the cyclic group $\matt{*}{0}{0}{1}$ of order $q-1$,
\item the dihedral group $\matt{*}{0}{0}{1} \cup \matt{0}{1}{*}{0}$
  of order $2(q-1)$ and, if~$q$ is odd, its two dihedral subgroups
\begin{gather*}
  \Big\{\matt{a}{0}{0}{1} \mid \text{$a \ne 0$ a square}\Big\} \cup \Big\{\matt{0}{1}{c}{0} \mid \text{$c \ne 0$ a square}\Big\} \quad \text{and} \\
  \Big\{\matt{a}{0}{0}{1} \mid \text{$a \ne 0$ a square}\Big\} \cup \Big\{\matt{0}{1}{c}{0} \mid \text{$c$ not a square}\Big\},
\end{gather*}
both of order $q-1$,
\item the Borel subgroup $\matt{*}{*}{0}{1}$ of order $q^2-q$,
\item if $q$ is odd, $\PSL_2(\Fo)$ of index $2$,
\item if $q=q'^2$ is a square, $\PGL_2(\F_{q'})$ of order $q'^3-q'=q'(q-1)$,
and
\item $\PGL_2(\Fo)$.
\end{enumerate}

In the last case $P$ is absolutely irreducible, thus it remains to investigate the first five cases which are treated in the next subsection.

Remark: The condition $q>61$ rules out some small subgroups as $A_4$, $S_4$, and~$A_5$.  In many of the finitely many cases $q \le 61$ the proof of the theorem also works (e.g., $q$ not a square and $q-1 \nmid 120$).  The condition of $q$ not being a power of even exponent of $2$ eliminates the fifth case in characteristic $2$; removing this condition would be of some interest.

\subsection{The individual cases}

Since the stabilisers $S_i$ are conjugates of each other,
one can assume without loss of generality that $S_1$ is one of the
explicit subgroups given in the previous subsection.
Then the polynomial $P_1$ is invariant under certain transformations
of $U$, so that $P_1$ and $P$ can be rewritten in terms of another
variable as stated in the following.

If a polynomial (in the variable $U$) is invariant under
$U \mapsto aU$, $a \in \F_q^{\times}$, it can be considered as a polynomial
in the variable $V=U^{q-1}$.
For the polynomials~$D$ and~$E^{q-1}$ one obtains
$$
D=\frac{V^{q+1}-1}{V-1} \qquad \text{and} \qquad
E^{q-1}=V(V-1)^{q-1}.
$$

Similarly, in the case of odd $q$, if a polynomial is invariant under 
$U \mapsto aU$ for all squares $a \in \F_q^{\times}$, it can be
rewritten in the variable $V'=U^{\frac{q-1}2}$.
For $D$ and $E^{q-1}$ this gives
$$
D=\frac{V'^{2q+2}-1}{V'^2-1} \qquad \text{and} \qquad
E^{q-1}=V'^2(V'^2-1)^{q-1}.
$$

If a polynomial is invariant under $U \mapsto U+b$, $b \in \Fo$,
it can be considered as a polynomial in $\tilde{V}=U^q-U$ which gives
$$
D=\tilde{V}^{q-1}+1 \qquad \text{and} \qquad
E^{q-1}=\tilde{V}^{q-1}.
$$

Combining the above yields that a polynomial which is invariant under both
$U \mapsto aU$, $a \in \F_q^{\times}$, and
$U \mapsto U+b$, $b \in \Fo$, can be considered as a polynomial
in $W=\tilde{V}^{q-1}=(U^q-U)^{q-1}$.
For $D$ and $E^{q-1}$ one obtains
$$
D=W+1 \qquad \text{and} \qquad
E^{q-1}=W.
$$

This is now applied to the various cases for $S_1$.

\subsubsection{The cyclic case}

Rewriting $P$ and $P_1$ in terms of $V=U^{q-1}$ one obtains
$$
P = \Big( \frac{V^{q+1}-1}{V-1} \Big)^{q+1}F^q-V^q(V-1)^{q^2-q}G^{q+1}
$$
\sloppy and $\deg_V(P_1)=1$, i.e., $P_1=p_1V-p_0$ with
$p_i \in \overline{K}[A]$, $\gcd(p_0,p_1)=1$,
$\max(\deg(p_0),\deg(p_1))=1$ and
it can be assumed that $p_0$ is monic.

The divisibility $P_1 \mid P$ transforms into the following
polynomial identity in $\overline{K}[A]$:
$$
\Big( \frac{p_0^{q+1}-p_1^{q+1}}{p_0-p_1} \Big)^{q+1} F^q =
p_1^q p_0^q (p_0-p_1)^{q^2-q} G^{q+1}.
$$
The degree of the first factor on the left hand side is either
$q^2+q$ or $q^2-1$ (if $p_0-\zeta p_1$ is constant for some
$\zeta \in \mu_{q+1}(\Ft) \setminus \{1\}$).
Since the degrees of the other factors are all divisible by $q$, the
latter case is impossible.
Since $\deg(F) = 2$ one gets $\deg(F^q) = 2q$.
Furthermore, $\deg((p_0p_1)^q) \in \{q,2q\}$,
$\deg((p_0-p_1)^{q^2-q}) \in \{0,q^2-q\}$ and $\deg(G^{q+1})=q^2+q$ which
implies $\deg(p_0-p_1)=0$, $\deg(p_0)=\deg(p_1)=1$ since $q>2$.

Let $p_0-p_1=c_1 \in \overline{K}$; in the following $c_i$ will be
some constants in $\overline{K}$.
Since the first factor on the left hand side is coprime to
$p_0p_1$, it follows
$$
\frac{p_0^{q+1}-p_1^{q+1}}{p_0-p_1} = c_2G,\quad F=c_3p_0p_1
\quad \text{and} \quad c_2^{q+1}c_3^q=c_1^{q^2-q}.
$$
Exchanging $\rho_1$ and $\rho_2$, if needed, one obtains
$$
p_0=A-\rho_1,\quad p_1=A-\rho_2,\quad c_3=\alpha
\quad \text{and} \quad c_1=\rho_2-\rho_1.
$$
Considering the coefficient of $A^q$ in the equation for $G$ gives $c_2=1$
and evaluating this equation at $A=\rho_2$ gives
$$
\rho_1^q+\alpha \rho_2 + \delta=0.
$$
This means that condition $(*)$ does not hold.

\subsubsection{The dihedral cases}

The case of the dihedral group of order $2(q-1)$ is considered first.
Then, as above, $P$ and $P_1$ can be expressed in terms of $V$, and,
since $P$ and $P_1$ are also invariant under $V \mapsto \frac1V$, they
can be expressed in terms of $W_+=V+\frac1V$.
This gives $\deg_{W_+}(P_1)=1$
and with ${\mathcal Z}=\mu_{q+1}(\Ft) \setminus \{1\}$
$$
D^{q+1}V^{-\frac{q^2+q}2}=
\prod_{\zeta \in {\mathcal Z}}
(W_+ -(\zeta+\zeta^q))^{\frac{q+1}2} \qquad \text{and}
$$
$$
PV^{-\frac{q^2+q}2}=
\Big( \prod_{\zeta \in {\mathcal Z}}
(W_+ -(\zeta+\zeta^q))^{\frac{q+1}2} \Big) F^q
-(W_+ -2)^{\frac{q^2-q}2}G^{q+1}.
$$
In characteristic $2$ each factor of the product over ${\mathcal Z}$
appears twice, thus justifying their exponent~$\frac{q+1}2$.

By writing $P_1=p_1W_+ -p_0$, with $p_i \in \overline{K}[A]$, $\gcd(p_0,p_1)=1$,
$\max(\deg(p_0),\deg(p_1))=2$ and~$p_0$ being monic,
the divisibility $P_1 \mid P$ transforms into the following
polynomial identity in~$\overline{K}[A]$:
$$
\Big( \prod_{\zeta \in {\mathcal Z}}
(p_0 -(\zeta+\zeta^q)p_1)^{\frac{q+1}2} \Big) F^q =
p_1^q (p_0-2p_1)^{\frac{q^2-q}2} G^{q+1}.
$$
Again the degree of the first factor on the left hand side must be divisible
by $q$ (respectively, $\frac{q}2$ in characteristic $2$),
and since $p_0-(\zeta+\zeta^q)p_1$ can be constant or linear for
at most one sum $\zeta+\zeta^q$,
the degree of the first factor must be $q^2+q$ for $q>4$.
Also the degree of $p_0-2p_1$ must be zero since $q>3$ and
thus the degree of $p_1$ is~$2$.

In even characteristic $p_0-2p_1=p_0$ is a constant, thus $p_0=1$
($p_0$ is monic).
The involution $\zeta \mapsto \zeta^q=\zeta^{-1}$ on ${\mathcal Z}$
has no fixed points, and, denoting by ${\mathcal Z}_2$ a set of representatives
of ${\mathcal Z}$ modulo the involution, one obtains
$$
\prod_{\zeta \in {\mathcal Z}_2} (1-(\zeta+\zeta^q)p_1)=c_1G,\quad
F=c_2p_1
\quad \text{and} \quad c_1^{q+1}c_2^q=1.
$$
Modulo $F$ one gets $F \mid c_1G-1$ which implies $c_1 \in K$.
Thus $c_2 \in K$, $p_1 \in K[A]$ and therefore
$P_1 \in K[A,U]$.

In odd characteristic the factor corresponding to $\zeta=-1$,
namely $(p_0+2p_1)^{\frac{q+1}2}$, is coprime to the other factors
in the product and coprime to $p_1(p_0-2p_1)$.
Hence $p_0+2p_1$ must be a square and its square root must divide $G$.
Moreover, one gets $F=c_1p_1$.
Since $p_0-2p_1=c_2$ is a constant and $p_0$ is monic,
one gets $c_1=2\alpha$, implying $p_1 \in K[A]$.
Since $p_0+2p_1=4p_1+c_2$ is a square, its discriminant is zero,
thus $c_2 \in K$ and hence $P_1 \in K[A,U]$.

If $S_1$ is one of the two dihedral subgroups of order $q-1$ 
(which implies that $q$ is odd), the argumentation is similar.
The polynomials $P$ and $P_1$ are expressed in terms of $V'=U^{\frac{q-1}2}$
and then,
since $U \mapsto \frac1{cU}$ becomes $V' \mapsto c^{-\frac{q-1}2}\frac1{V'}$
with $c^{-\frac{q-1}2}=\pm 1$,
in terms of $W'_+=V'+\frac1{V'}$ or $W'_-=V'-\frac1{V'}$,
respectively.
In the first case $P$ is rewritten as
$$
PV'^{-(q^2+q)}=
\Big( \prod_{\zeta \in {\mathcal Z}'}
(W'_+ -(\zeta+\zeta^{-1}))^{\frac{q+1}2} \Big) F^q
-(W'_+ -2)^{\frac{q^2-q}2}(W'_+ +2)^{\frac{q^2-q}2}G^{q+1}
$$
where ${\mathcal Z}'=\mu_{2(q+1)}(\Ft) \setminus \{\pm 1\}$.
By setting $P_1=p_1W'_+ -p_0$ with $p_i \in \overline{K}[A]$,
$\gcd(p_0,p_1)=1$, $\max(\deg(p_0),\deg(p_1))=1$ and $p_0$ being monic,
one obtains
$$
\Big( \prod_{\zeta \in {\mathcal Z}'}
(p_0 -(\zeta+\zeta^{-1})p_1)^{\frac{q+1}2} \Big) F^q =
p_1^{2q} (p_0-2p_1)^{\frac{q^2-q}2} (p_0+2p_1)^{\frac{q^2-q}2} G^{q+1}.
$$
Since one of $p_0 \pm 2p_1$ is not constant, the degree of the right
hand side exceeds the degree of the left hand side for $q>5$ which
is a contradiction.

In the second case $P$ is rewritten as
$$
PV'^{-(q^2+q)}=
\Big( \prod_{\zeta \in {\mathcal Z}'}
(W'_- -(\zeta-\zeta^{-1}))^{\frac{q+1}2} \Big) F^q
-W_-'^{q^2-q} G^{q+1}
$$
and by setting $P_1=p_1W'_- -p_0$ with $p_i \in \overline{K}[A]$,
$\gcd(p_0,p_1)=1$, $\max(\deg(p_0),\deg(p_1))=1$ and $p_0$ being monic,
one obtains
$$
\Big( \prod_{\zeta \in {\mathcal Z}'}
(p_0 -(\zeta-\zeta^{-1})p_1)^{\frac{q+1}2} \Big) F^q =
p_1^{2q} p_0^{q^2-q} G^{q+1}.
$$
Considering the degrees for $q>3$ it follows that $p_0$ must be constant
and hence $p_1$ is of degree one.
Since $p_1$ is coprime to the first factor on the left hand side,
it must divide $F^q$ which implies $\rho_1=\rho_2 \in K$, contradicting the irreducibility of~$F$.

\subsubsection{The Borel case}

In this case, rewriting $P$ and $P_1$ in terms of $W=(U^q-U)^{q-1}$
gives
$$
P=(W+1)^{q+1}F^q-W^qG^{q+1}
$$
and $\deg_W(P_1)=1$, $P_1=p_1W-p_0$, with $p_i \in \overline{K}[A]$,
$\gcd(p_0,p_1)=1$, $\max(\deg(p_0),\deg(p_1))=q$ and $p_1$ being monic.
Then the divisibility $P_1 \mid P$ transforms into the following
polynomial identity in $\overline{K}[A]$:
$$
(p_0+p_1)^{q+1} F^q = p_1 p_0^q G^{q+1}.
$$
From $\deg(G^{q+1})=q^2+q$, $\deg(p_1 p_0^q)\ge q$ and $\deg(F^q) = 2q$
it follows that the degree of $p_0+p_1$ must be $q$.
This implies $\deg(F^q)=\deg(p_1 p_0^q)$, thus $\deg(p_0) \le 2$ and
therefore $\deg(p_1)=q$,  since $q>2$, and $\deg(p_0) = 1$.

Since $p_0+p_1$ is coprime to $p_0p_1$, it follows
$$
p_0+p_1 = c_1G,\quad p_1=\tilde{p}^q,\quad F=c_2\tilde{p}p_0
\quad \text{and} \quad c_1^{q+1}c_2^q=1
$$
for a monic linear polynomial $\tilde{p} \in \overline{K}[A]$.

Exchanging $\rho_1$ and $\rho_2$, if needed, one obtains
$$
\tilde{p}=A-\rho_1,\quad p_0=c_3(A-\rho_2),\quad c_1=1,\quad c_2=1
\quad \text{and} \quad c_3=\alpha.
$$
Evaluating $p_0+p_1 = G$ at $A=0$ gives
$$
\rho_1^q + \alpha \rho_2 + \delta = 0.
$$
This means that condition $(*)$ does not hold.

\subsubsection{The $\PSL_2$ case}

This case can only occur for odd $q$, and
then $P$ splits as $P=sP_1P_2$ with a scalar $s \in \overline{K}$.
The map $U \mapsto aU$ for a non-square $a \in \Fo$
exchanges $P_1$ and $P_2$.
Since $\PSL_2(\Fo)$ is a normal subgroup of $\PGL_2(\Fo)$, $P_2$ is
invariant under $\PSL_2(\Fo)$ as well.
By rewriting $P$ in terms of $W'=(U^q-U)^{\frac{q-1}2}$ one obtains
$$
P=(W'^2+1)^{q+1}F^q-W'^{2q}G^{q+1}=sP_1(W')P_1(-W').
$$
Denoting by $p_0 \in \overline{K}[A]$ the constant coefficient of 
$P_1 \in (\overline{K}[A])[W']$ this becomes modulo $W'$
$$
F^q=sp_0^2
$$
which implies $\rho_1=\rho_2 \in K$, contradicting the irreducibility of~$F$.

\subsubsection{The case $\PGL_2(\F_{q'})$}

Since $\PGL_2(\F_{q'}) \subset \PSL_2(\Fo)$ in odd characteristic,
one can reduce this case to the previous case as follows.

Let $I_1 \subset \{1,\ldots,g\}$ be the subset of $i$ such that
$S_i$ is a conjugate of $S_1$ by an element in $\PSL_2(\Fo)$, and let
$I_2=\{1,\ldots,g\} \setminus I_1$.
These two sets correspond to the two orbits of the action of
$\PSL_2(\Fo)$ on the $S_i$ (or $P_i$).
Both orbits contain $\#I_1=\#I_2=\frac{g}2$ elements and an
element in $\PGL_2(\Fo) \setminus \PSL_2(\Fo)$ transfers one orbit
into the other.

Let $\tilde{P}_j=\prod_{i \in I_j} P_i$, $j=1,2$, then
$P$ splits as $P=s\tilde{P}_1\tilde{P}_2$, $s \in \overline{K}$, and both
$\tilde{P}_j$, $j=1,2$, are invariant under $\PSL_2(\Fo)$.
Notice that the absolute irreducibility of $P_1$ and $P_2$ was not
used in the argument in the $\PSL_2$ case.

This completes the proof of Theorem~\ref{thm:heu2}.

\subsection{Traps}

In the following Proposition~\ref{prop:cond} and Proposition~\ref{prop:trap} are proven.

Let~$Q$ be an irreducible quadratic polynomial in $K[X]$ such that $(1, {u_0 X + u_1}), (X, {v_0 X + v_1})$ is a basis of the lattice $L_Q$, so that $Q$ is a scalar multiple of $-u_0 X^2 + (-u_1 + v_0) X + v_1 = F(-X)$ and has roots~$-\rho_1$ and~$-\rho_2$.
By definition of $L_Q$ the pair $(h_0,h_1)$ must be in the dual lattice (scaled by $Q$), given by the basis $(u_0 X + u_1, -1), (v_0 X + v_1, -X)$.

For the assertions concerning conditions $(*)$ and $(**)$, assume that $\rho_1,\rho_2 \in L \setminus K$ and that
\[ \rho_1^q + \alpha \rho_j + \delta = 0 \]
holds for $j=1$ or $j=2$.

First consider the case $j=2$, i.e., condition $(*)$.  To show that $-\rho_i$, $i=1,2$, are roots of $h_1 X^q - h_0$ it is sufficient to show this for the basis of the dual lattice of $L_Q$ given above.
For $(u_0 X + u_1, -1)$ one computes
\[ -(-\rho_1^q) - u_0 (-\rho_1) - u_1 = \rho_1^q - \alpha \rho_1 - \beta + \delta = -\alpha \rho_2 - \alpha \rho_1 - \beta = 0 , \]
and for $(v_0 X + v_1, -X)$ one obtains
\[ -(-\rho_1) (-\rho_1^q) - v_0 (-\rho_1) - v_1 = (-\rho_1^q - \delta) \rho_1 - \gamma = \alpha \rho_1 \rho_2 - \gamma = 0 . \]
Therefore $h_1 X^q - h_0$ is divisible by~$Q$, which is then a trap of level $0$.

In the case $j=1$ an analogous calculation shows that $-\rho_i$, $i=1,2$, are roots of $h_1 X^{q^{kd+1}} - h_0$, namely for $(u_0 X + u_1, -1)$ one has
\[ -(-\rho_2^{q^{kd+1}}) - u_0(-\rho_2) - u_1 = \rho_1^q - \alpha \rho_2 - \beta + \delta = -\alpha \rho_1 - \alpha \rho_2 - \beta = 0 \]
and for $(v_0 X + v_1, -X)$ one gets
\[ -(-\rho_2) (-\rho_2^{q^{kd+1}}) - v_0(-\rho_2) - v_1 = (-\rho_1^q - \delta) \rho_2 - \gamma = \alpha \rho_1 \rho_2 - \gamma = 0 \]
Therefore $h_1 X^{q^{kd+1}} - h_0$ is divisible by~$Q$, which is then a trap of level~$kd$.  This finishes the proof of Proposition~\ref{prop:cond}.

Regarding Proposition~\ref{prop:trap}, note that a solution $(a,B)$ gives rise to
the polynomial ${\mathcal P}_a=a(u_0X+(Y+u_1))+((Y+v_0)X+v_1)$.
If, for $j=1$ or $j=2$, $\rho$ is a root of ${\mathcal P}_a \bmod f_j$ for two different values of $a$, then $\rho$ is a root of $u_0X+(Y+u_1) \bmod f_j$ and of $(Y+v_0)X+v_1 \bmod f_j$.  Since
\[ -X(u_0X+(Y+u_1)) + (Y+v_0)X+v_1 = -u_0X^2 + (-u_1+v_0)X +v_1 = F(-X), \]
which equals~$Q$ up to a scalar, it follows that $\rho$ is also a root of~$Q$.  Furthermore, in the case $j=1$ the polynomial ${\mathcal P}_a \bmod f_1$ splits completely, so that $\rho \in K$, contradicting the irreducibility of $Q$,
finishing the proof of Proposition~\ref{prop:trap}.  

This completes the proof of Theorem~\ref{thm-onthefly}.

%-----------------------------------------------------------------------------

\section*{Acknowledgements}

The authors are indebted to Claus Diem for explaining how one can obviate the need to compute the logarithms of the factor base elements, and wish to thank him also for some enlightening discussions.

\bibliographystyle{plain}
\bibliography{dl_bib}

\end{document}